\documentclass[12pt]{article}
\usepackage{fullpage}
\usepackage{amssymb}
\usepackage{amsthm}
\usepackage{graphicx}
\usepackage[all,2cell]{xy}
\usepackage{verbatim}
\newtheorem{theorem}{Theorem}[section]
\newtheorem{corollary}[theorem]{Corollary}
\newtheorem{definition}[theorem]{Definition}

\newtheorem{lemma}[theorem]{Lemma}

\begin{document}
\title{Torus invariants of the homogeneous coordinate ring of $G/B$ -  connection with Coxeter elements }

\author{S. Senthamarai Kannan, B. Narasimha Chary, Santosha Kumar Pattanayak}
\maketitle

Chennai Mathematical Institute, Plot H1, SIPCOT IT Park, 

Siruseri, Kelambakkam  603103, Tamilnadu, India.

E-mail:kannan@cmi.ac.in, chary@cmi.ac.in, santosh@cmi.ac.in.

\begin{abstract}
In this article, we prove that for any indecomposable dominant character $\chi$ of a maximal torus $T$ of
 a simple adjoint group $G$ over $\mathbb C$ such that there is a Coxeter element $w$ in the Weyl group $W$ for which 
$X(w)^{ss}_T(\mathcal L_\chi) \neq \emptyset$,  the graded algebra 
$\oplus_{d \in \mathbb Z_{\geq 0}}H^0(G/B, \mathcal L_\chi^{\otimes d})^T$ is a 
polynomial ring if and only if  $dim(H^0(G/B , \mathcal L_{\chi})^T) \leq$ rank of $G$.
We also prove that  the co-ordinate ring $\mathbb C[\mathfrak{h}]$ of the cartan subalgebra $\mathfrak h$ of the Lie algebra $\mathfrak g$ of $G$ and 
$\oplus_{d \in \mathbb Z_{\geq 0}}H^0(G/B,\mathcal L_{\alpha_0}^{\otimes d})^T$ are isomorphic if 
and only if $X(w)_{T}^{ss}(\mathcal L_{\alpha_0})$ is non empty for some coxeter element $w$ in $W$, 
where $\alpha_0$ denotes the highest long root.
\end{abstract}


\section{Introduction}

In \cite{r1},\cite{r10},\cite{r12} Chevalley, Serre, Shephard-Todd have proved that for any faithful 
representation $V$ of a finite group $H$ over the field $\mathbb C $ of complex numbers, the ring of $H$
invariants $\mathbb C[V]^{H}$ is a polynomial algebra if and only if $H$ is 
generated by pseudo reflections in $GL(V)$.

Chevalley also proved that for any semisimple algebraic group
$G$ over $\mathbb C$, the ring  $\mathbb C[\mathfrak{g}]^{G}$ of $G$ invariants of the 
co-ordinate ring $\mathbb C[\mathfrak{g}]$ of the adjoint representation $\mathfrak{g}$ of $G$
is a polynomial algebra(refer to page 127 in \cite{r3}).

In \cite{r13}, Steinberg proved that for any semisimple simply connected algebraic 
group $G$ (over any algebraically closed field $K$) acting on itself by
conjugation, the ring of $G$-invariants $K[G]^{G}$ is a polynomial algebra (refer to theorem in page 41 of \cite{r13}).

In\cite{r14}, D. Wehlau proved a theorem giving a  neccessary and sufficient condition for a rational 
representation $V$ of a torus $S$ for which the ring of $S$ invariants of 
the co-ordinate ring $K[V]$ is a polynomial algebra (refer to theorem 5.8 of \cite{r14}).  

We now set up some notation before proceeding further.

Let $G$ be a simple adjoint group of rank $n$
over the field of complex numbers.  Let $T$ be a maximal torus of $G$, 
$B$ be a Borel subgroup of $G$ containing $T$. 
Let $N_{G}(T)$ denote the normaliser of $T$ in $G$. Let $W=N_{G}(T)/T$
denote the Weyl group of $G$ with respect to $T$.

We denote by $\mathfrak{g}$ the Lie algebra of $G$. 
We denote by $\mathfrak{h}\subseteq \mathfrak{g}$ the Lie algebra of $T$.
Let $R$ denote the roots of $G$ with respect to $T$.
Let $R^{+}\subset R$ be the set of positive roots with respect to $B$.
Let $S=\{\alpha_{1}, \alpha_{2}, \cdots ,\alpha_n\}\subset R^{+}$ denote 
the set of simple roots with respect to $B$.
Let $\langle . , . \rangle$ denote the restriction of the Killing form to 
$\mathfrak{h}$.
Let $s_{i}$ denote the simple reflection corresponding to the simple root $\alpha_{i}$.

A element $w$ in $W$ is said to a Coxeter element if it has a reduced expression of the form $s_{i_{1}}s_{i_{2}} \cdots s_{i_{n}}$
such that $i_{j}\neq i_{k}$ whenever $j\neq k$ ( refer to \cite{r2}).

We denote by $G/B$, the flag variety of all Borel subgroups of $G$.
For any $w\in W$, we denote by $X(w)=\overline{BwB/B}\subset G/B$ the 
Schubert Variety corresponding to $w$. We note that $X(w)$ is stable 
for the left action of $T$ on $G/B$.

 We denote $\mathcal L_{\chi}$, 
the line bundle associated to a character $\chi$ of $T$.

$X(w)^{ss}_T(\mathcal L_\chi)$ denote the set of all semistable points with respect 
to the line bundle $\mathcal L_\chi$ and for the action of $T$(refer to \cite{r8},\cite{r9}).

In \cite{r5} and \cite{r6} some properties of semi stable points $X(w)_{T}^{ss}(\mathcal L_{\chi})$ with respect to dominant characters $\chi$ of $T$ are studied. For instance in \cite{r6}, for every simple algebraic group $G$ all the Coxeter elements $w$ in the Weyl group $W$ for 
which  $X(w)^{ss}_T(\mathcal L_\chi)$ is non empty  for some  non trivial dominant character $\chi$ of $T$ has been studied.

A further study about the dominant characters $\chi$ of $T$ for which there is a Coxeter element $w$ such that 
$X(w)_{T}^{ss}(\mathcal L_{\chi})$ is non empty, we observed that in the case of $A_2$, given an indecomposable dominant character $\chi$ of $T$ which is in the root lattice (refer to definition 4.1  for the indecomposable dominant character ), $X(w)_{T}^{ss}(\mathcal L_{\chi})$ is non empty for some coxeter element $w$ in $W$ if and only if $\chi$ must be one of the following: $\alpha_{1}+\alpha_{2}$, 
$2\alpha_{1}+\alpha_{2}$ and $\alpha_{1}+2\alpha_{2}$. We also observed  that for all these three dominant characters $\chi$, the ring of $T$-invariants of 
the homogeneous co-ordinate ring 
$\bigoplus_{d\in \mathbb Z_{\geq 0}}H^{0}(\mathbb PGL_3(\mathbb C)/B, \mathcal L_{\chi}^{\otimes d})$  
 is a polynomial ring.

 In case of $B_2$ as well, given an indecomposable dominant character $\chi$ of $T$ which is in the root lattice, $X(w)_{T}^{ss}(\mathcal L_{\chi})$ is non empty for some coxeter element $w$ in $W$ if and only if $\chi$ must be one of the following: $\alpha_{1}+\alpha_{2}$, 
$\alpha_{1}+2\alpha_{2}$. We also observed that for all these two dominant characters $\chi$, the ring of $T$-invariants of 
the homogeneous co-ordinate ring 
$\bigoplus_{d\in \mathbb Z_{\geq 0}}H^{0}( SO(5,\mathbb C)/B, \mathcal L_{\chi}^{\otimes d})$  
 is a polynomial ring.

The computations in the above mentioned special cases tempt us to ask the following question:

Let $G$ be a simple adjoint group over $\mathbb C$, the field of complex numbers. Let $T$ be a maximal torus of $G$, $B$ be a Borel subgroup 
of $G$ containing $T$. Then, for any indecomposable dominant character $\chi$ of 
$T$ such that there is a Coxeter element $w$ in $W$ such that 
$X(w)_{T}^{ss}(\mathcal L_{\chi})$ is non empty, is the ring of $T$- invariants of the 
homogeneous co-ordinate ring 
$\bigoplus_{d\in \mathbb Z_{\geq 0}}H^{0}(G/B, \mathcal L_{\chi}^{\otimes d})$ a polynomial 
algebra ?

In this paper we prove that for any indecomposable dominant character$\chi$ of a maximal torus $T$ of a simple adjoint group $G$ such that there is a Coxeter element $w \in W$ for which $X(w)^{ss}_T(\mathcal L_\chi) \neq \emptyset$,  the graded algebra 
$\oplus_{d \in \mathbb Z_{\geq 0}}H^0(G/B, \mathcal L_\chi^{\otimes d})^T$ is a 
polynomial ring if and only if  $dim(H^0(G/B , \mathcal L_{\chi})^T) \leq$ rank of $G$ 
(refer to theorem 4.8).

Let $\mathfrak g$ be the Lie algebra of $G$, let $\mathfrak h$ be the Lie algebra of $T$,
let $\alpha_0$ be the highest long root. Since 
$H^0(G/B, \mathcal L_{\alpha_0})$ is an irreducible self dual $G$ module with highest weight $\alpha_0$, the $G$ modules $H^0(G/B, \mathcal L_{\alpha_0})$, 
$Hom(\mathfrak g, \mathbb C)$ are isomorphic.

On the other hand, the natural $T$-invariant projection from $\mathfrak{g}$ to $\mathfrak{h}$ induces a isomorphism $Hom(\mathfrak h, \mathbb C)\rightarrow Hom(\mathfrak g, \mathbb C)^T $ 
So, we have  an isomorphism $Hom(\mathfrak h, \mathbb C)\rightarrow H^0(G/B, \mathcal L_{\alpha_0})^T$ .

Thus we have a homomorphism $f: \mathbb C[\mathfrak{h}] \longrightarrow \oplus_{d \in \mathbb Z_{\geq 0}}H^0(G/B,\mathcal L_{\alpha_0}^{\otimes d})^T$ of $\mathbb C$ algebras.

In this direction, we prove that the  homomorphism $f: \mathbb C[\mathfrak{h}] \longrightarrow \oplus_{d \in \mathbb Z_{\geq 0}}H^0(G/B,\mathcal L_{\alpha_0}^{\otimes d})^T$ is an isomorphism if and only if $X(w)_{T}^{ss}(\mathcal L_{\alpha_0})$ is non empty for some coxeter element $w$ in $W$ (refer to theorem 3.3).\\

The organisation of the paper is as follows:

Section 2 consists of  preliminaries and notations.

In section 3, we prove that  the homomorphism $f: \mathbb C[\mathfrak{h}] \rightarrow  \oplus_{d \in \mathbb Z_{\geq 0}}H^0(G/B,\mathcal L_{\alpha_0}^{\otimes d})^T$ is an isomorphism if and only if $X(w)_{T}^{ss}(\mathcal L_{\alpha_0})$ is non empty for some coxeter element $w$ in $W$.

In section 4,  we prove that for any indecomposable dominant character$\chi$ of a maximal torus $T$ of a simple adjoint group $G$ over $\mathbb C$ such that there is a Coxeter element $w \in W$ for which $X(w)^{ss}_T(\mathcal L_\chi) \neq \emptyset$,  the graded algebra 
$\oplus_{d \in \mathbb Z_{\geq 0}}H^0(G/B, \mathcal L_\chi^{\otimes d})^T$ is a 
polynomial ring if and only if  $dim(H^0(G/B , \mathcal L_{\chi})^T) \leq$ rank of $G$

\section{Preliminaries and Notations}

Let $G$ be a simple adjoint group of rank $n$
over the field of complex numbers.  Let $T$ be a maximal torus of $G$, 
$B$ be a Borel subgroup of $G$ containing $T$. 
Let $N_{G}(T)$ denote the normaliser of $T$ in $G$. Let $W=N_{G}(T)/T$
denote the Weyl group of $G$ with respect to $T$.

We denote by $\mathfrak{g}$ the Lie algebra of $G$. 
We denote by $\mathfrak{h}\subseteq \mathfrak{g}$ the Lie algebra of $T$.
Let $R$ denote the roots of $G$ with respect to $T$.
Let $R^{+}\subset R$ be the set of positive roots with respect to $B$.
Let $S=\{\alpha_{1}, \alpha_{2}, \cdots ,\alpha_n\}\subset R^{+}$ denote 
the set of simple roots with respect to $B$.
Let $\langle . , . \rangle$ denote the restriction of the Killing form to 
$\mathfrak{h}$.
Let $\check{\alpha}_i$ denote the co-root corresponding to $\alpha_{i}$.
Let $\varpi_1, \varpi_2,\cdots, \varpi_n$ fundamental weights corresponding to $S$.
Let $s_{i}$ denote the simple reflection corresponding to the simple root $\alpha_{i}$.
For any subset $J$ of $\{1,2,\cdots,n\}$, we denote by $W_J$ the subgroup of $W$ generated by $s_j$, $j \in J$. 
We denote $J^c$ is the complement of $J$ in $\{1,2,\cdots,n\}$.
We denote $P_J$ the parabolic subgroup of $G$ containing $B$ and $W_{J^c}$.
In particular, we denote the maximal parabolic subgroup of $G$ generated by $B$ and $\{s_j ; j \neq i \}$ by $P_i$
(refer to \cite {r4}).

Let $w_0$ denote the longest element of $W$ corresponding to $B$.
Let $B^-=w_0Bw_0^{-1}$ denote the Borel subgroup of $G$ opposite to $B$ with respect to $T$.

Let $\{E_{\beta}: \beta \in R\}\bigcup \{H_{\beta}: \beta \in S \} $ be the 
Chevalley basis  for $\mathfrak{g}$ (refer to \cite {r3}).  

A element $w$ in $W$ is said to a Coxeter element if it has a reduced expression of the form $s_{i_{1}}s_{i_{2}} \cdots s_{i_{n}}$
such that $i_{j}\neq i_{k}$ whenever $j\neq k$ ( refer to \cite{r2}).

We denote by $G/B$, the flag variety of all Borel subgroups of $G$.
For any $w\in W$, we denote by $X(w)=\overline{BwB/B}\subset G/B$ the 
Schubert Variety corresponding to $w$. We note that $X(w)$ is stable 
for the left action of $T$ on $G/B$.

Let $\chi \in 
X(B)$,   we 
have an action of $B$ on $\mathbb C$, namely $b.k=\lambda(b^{-1})k, \,\, b \in B, 
\,\, k \in \mathbb C$. $L:= G \times \mathbb C/ \sim$, where $\sim$ is the equivalence 
relation defined by $(gb, b.k) \sim (g, k), g \in G, b \in B, k \in \mathbb C$. Then $L$ is the 
total space of a line bundle over $G/B$. We denote by $\mathcal L_{\chi}$, 
the line bundle associated to $\chi$.

$X(w)^{ss}_T(\mathcal L_\chi)$ denote the set of all semistable points with respect 
to the line bundle $\mathcal L_\chi$ and for the action of $T$(for precise definition, refer to \cite{r8},\cite{r9}).

Let $\mathfrak{g}=Lie(G)$ be the adjoint representation of $G$ and 
$\mathfrak{h}=Lie(T)$.
Let $\alpha_0$ be the highest long root. 

Since $G$ is simple, the adjoint representation $\mathfrak{g}$ of $G$ is an 
irreducible representation with highest weight $\alpha_{0}$.

Let $\phi_{1}:\mathfrak{g} \rightarrow \mathfrak{h}$ be the $T$-invariant projection. Then, 
$\phi_1$ induces a natural isomorphism $Hom(\mathfrak h, \mathbb C) \rightarrow Hom(\mathfrak g, \mathbb C)^T$. 

Since $H^0(G/B, \mathcal L _{\alpha_0})$ is an irreducible self dual $G$ module with highest weight $\alpha_0$, the $G$ modules $H^0(G/B, \mathcal L_{\alpha_0})$, 
$Hom(\mathfrak g, \mathbb C)$ are isomorphic. 

So, we have  an isomorphism $Hom(\mathfrak h, \mathbb C)\rightarrow H^0(G/B, \mathcal L_{\alpha_0})^T$ .

{\it Thus we have a homomorphism $f: \mathbb C[\mathfrak{h}] \longrightarrow \oplus_{d \in \mathbb Z_{\geq 0}}H^0(G/B,\mathcal L_{\alpha_0}^{\otimes d})^T$ of $\mathbb C$ algebras. (1)}

\section{Relationship between $\mathbb C [\mathfrak h]$ and homogeneous co-ordinate ring of $G/B$ associated to highest long root }

In this section,
{\it we  show that the homomorphism $f: \mathbb C[\mathfrak{h}] \longrightarrow \oplus_{d \in \mathbb Z_{\geq 0}}H^0(G/B,\mathcal L_{\alpha_0}^{\otimes d})^T$ as in (1) is injective}.

Further, we also prove that $f: \mathbb C[\mathfrak{h}] \rightarrow  \oplus_{d \in \mathbb Z_{\geq 0}}H^0(G/B,\mathcal L_{\alpha_0}^{\otimes d})^T$ is an isomorphism if and only if $X(w)_{T}^{ss}(\mathcal L_{\alpha_0})$ is non empty for some coxeter element $w$ in $W$.

We frist set up some notation.

For each positive root $\alpha\in R^{+}$, we denote by $U_{-\alpha}$,
the $T$- stable root subgroup of $B$ corresponding to $-\alpha$.
 
Let $U^{-}$ be the unipotet radical of the opposite Borel subgroup 
$w_{0}Bw_{0}^{-1}$. 

Then, we have  $U^-= \prod_{\alpha \in \Phi^+}U_{-\alpha}$.
 
Let $v^+=E_{\alpha_0}$ be a highest weight vector of $\mathfrak{g}$. 
Consider $U^-v^+ \subset \mathfrak{g}$ be the $U^-$orbit of $v^+$.  

\begin{lemma} The restriction map $\phi :=\phi_{1}|_{U^-v^+}:U^-v^+ \rightarrow \mathfrak{h}$ is onto.
\end{lemma}
\begin{proof}

Since $\alpha_0$ is dominant, we can choose a simple root $\gamma_1$ such 
that $\langle \alpha_0,\check{\gamma_1} \rangle \geq 1$. Choose distinct 
simple roots $\gamma_2, \gamma_3, \cdots ,\gamma_{n-1}$ such that for all 
$r=1,2 \cdots ,n-1$, $\sum_{j=1}^r\gamma_j$ is a root.

Denote $\theta_r=\sum_{j=1}^r \gamma_j$, $r=1,2, \cdots, n-1$. Again 
since $\langle \alpha_0, \check{\theta_r} \rangle \geq 1$ for 
$1 \leq r \leq n-1$, each $\beta_r:=\alpha_0-\theta_r$ is a root. 

For every choices of $c_0, c_r,c_r' \in \mathbb C$, $1 \leq r \leq n-1$ we claim 
that \\
 $\phi(exp(c_0E_{-\alpha_0}))(exp(c_1E_{-\beta_1}))(exp(c_1'E_{-\theta_1}))(exp(c_2E_{-\beta_2}))(exp(c_2'E_{-\theta_2})) 
\cdots (exp(c_{n-1}E_{-\beta_{n-1}}))$\\
$(exp(c_{n-1}'E_{-\theta_{n-1}}))(E_{\alpha_0})=-c_0H_{\alpha_0}-\sum_{r=1}^{n-1}c_rc_r'H_{\beta_r}.$

Take a typical monomial 
$M=\frac{c_0^{m_0}}{m_0!}E_{-\alpha_0}^{m_0}\frac{c_1^{a_1}}{a_1!}E_{-\beta_1}^{a_1}\cdots \frac{c_{n-1}^{a_{n-1}}}{a_{n-1}!}E_{-\beta_{n-1}}^{a_{n-1}}
\frac{(c'_1)^{b_1}}{b_1!}E_{-\theta_1}^{b_1}\cdots \frac{(c'_{n-1})^{b_{n-1}}}{b_{n-1}!}E_{-\theta_{n-1}}^{b_{n-1}}$ \\
occuring in the expansion of $(exp(c_0E_{-\alpha_0}))(exp(c_1E_{-\beta_1}))(exp(c_1'E_{-\theta_1}))(exp(c_2E_{-\beta_2}))
(exp(c_2'E_{-\theta_2})) \cdots$\\$
(exp(c_{n-1}E_{-\beta_{n-1}})(exp(c_{n-1}'E_{-\theta_{n-1}})$

Then $Mv^+$ has weight zero if and only if $(1-m_0)\alpha_0=\sum_{j=1}^{n-1}a_j\beta_j+\sum_{k=1}^{n-1}b_k\theta_k$.

\underline{Claim:} For all $j=1,2, \cdots ,n-1$ and $k=1,2, \cdots ,n-1$, 
there exist unique $r$ in $\{1,2, \cdots ,n-1\}$ such that $a_r=b_r=1$ and $a_j=b_j=0$ for all $j \neq r$. 

Now $Mv^+$ has weight $0$ implies $(1-m_0)\alpha_0-\sum_{j=1}^{n-1}a_j\beta_j-\sum_{k=1}^{n-1}b_k\theta_k=0$. \\
$\Rightarrow (m_0-1)\alpha_0+\sum_{j=1}^{n-1}a_j\beta_j+\sum_{k=1}^{n-1}b_k\theta_k=0$ \\
$\Rightarrow(m_0-1)\alpha_0+\sum_{j=1}^{n-1}a_j\beta_j+\sum_{k=1}^{n-1}b_k(\alpha_0-\beta_k)=0$ \\
$\Rightarrow(m_0-1)\alpha_0+(\sum_{k=1}^{n-1}b_k)\alpha_0+\sum_{j=1}^{n-1}a_j\beta_j-\sum_{k=1}^{n-1}b_k\beta_k=0$. \\
$\Rightarrow((m_0-1)+\sum_{k=1}^{n-1}b_k)\alpha_0+\sum_{j=1}^{n-1}(a_j-b_j)\beta_j=0$.

Since $\{ \alpha_0 , \beta_j, j=1,2,\cdots,n-1\}$ linearly independent, we have \\
$(m_0-1)+\sum_{k=1}^{n-1}b_k=0$ and $a_j=b_j$ for all $j=1,2, \cdots ,n-1$.

Since $m_0$ and $b_i's$ are non-negative integers, we have either $m_0=1$ and 
$a_j=b_j$ for all $j$ or $m_0=0$ and there exist unique $k$ such that $b_k=1$ and $b_j=0$ for all $j \neq k$. \\
Again $H_{\alpha_0}, H_{\beta_r}$ are linearly independent since 
$\alpha_0$ and $\beta_r$ are linearly independent.\\
So, we have a surjective map $ U^-v^+ \supseteq \mathbb (U_{-\alpha_0}\prod_{j=1}^{n-1}U_{-\beta_j}\prod_{k=1}^{n-1}U_{-\theta_k})v^+ \rightarrow  \mathfrak{h}$ 
given by\\ $(u_{-\alpha_0}(c_0)\prod_{j=1}^{n-1}u_{-\beta_j}(c_j)\prod_{k=1}^{n-1}u_{-\theta_k}(c_k'))v^+\mapsto -c_0H_{\alpha_0}-\sum_{r=1}^{n-1}c_rc_r'H_{\beta_r}$, \\
where $u_{\alpha}(c)=exp(cE_{\alpha})$, $\alpha \in R ,c \in \mathbb C$.\\
Hence $\phi: U^-v^+ \rightarrow \mathfrak{h}$ is onto. This completes the proof of the lemma.
\end{proof}

We have 

\begin{corollary}
The homomorphism $f: \mathbb C[\mathfrak{h}] \rightarrow  \oplus_{d \in \mathbb Z_{\geq 0}}H^0(G/B,\mathcal L_{\alpha_0}^{\otimes d})^T$ as in (1) is injective. 
\end{corollary}
\begin{proof}
By the lemma 3.1,  we have $\phi: U^-v^+ \rightarrow \mathfrak{h}$ is onto \\
So, $ \phi^*:\mathbb C[\mathfrak{h}] \rightarrow \mathbb C[U^-v^+]$ is injective.  \\
Since $U^-[v^+]$ is affine open subset of $G/B$ , we have restriction map $H^0($G/B$,\mathcal L_{\alpha_0} ^{\otimes d}) \longrightarrow \mathbb C [U^-v^+]$ 
 for all $d \in \mathbb Z_{\geq 0}$.\\
 So we get a map $H^0($G/B$,\mathcal L_\chi^{\otimes d})^T \longrightarrow \mathbb C [U^-v^+]$ 
 for all $d \in \mathbb Z_{\geq 0}$\\
 Hence we have a homomorphism $g: \oplus_{d \in \mathbb Z_{\geq 0}}H^0(G/B, \mathcal L_{\alpha_0}^{\otimes d})^T \longrightarrow \mathbb C [U^-v^+]$ of $\mathbb C$ algebras.
 
Now we have the following commutative diagram  

\[
\xymatrix{ 
\mathbb C[\mathfrak{h}] 
\ar[rr]^{f} 
\ar[dr]_{\phi^*} 
&& \oplus_{d \in \mathbb Z_{\geq 0}}H^0(G/B,\mathcal L_{\alpha_0}^{\otimes d})^T\ar[dl]^g \\
& \mathbb C[U^-v^+] }\\
\]
So, we have $g\circ f= \phi^*$.\\
 Hence the homomorphism $f:\mathbb C[\mathfrak{h}] \rightarrow  \oplus_{d \in \mathbb Z_{\geq 0}}H^0(G/B,\mathcal L_{\alpha_0}^{\otimes d})^T$ is injective. 
\end{proof}

We now prove the following theorem.

\begin{theorem} The homomorphism $f: \mathbb C[\mathfrak{h}] \rightarrow  \oplus_{d \in \mathbb Z_{\geq 0}}H^0(G/B,\mathcal L_{\alpha_0}^{\otimes d})^T$ is an isomorphism if and only if $X(w)_{T}^{ss}(\mathcal L_{\alpha_0})$ is non empty for some coxeter element $w$ in $W$.
\end{theorem}

\begin{proof}
By the theorem in 4.2 in \cite{r6}, $X(w)_{T}^{ss}(\mathcal L_{\alpha_0})$ is non empty for some coxeter element $w$ if and only if
$G$ is of type $A_n, B_2$ and $C_n$.

Now we prove that  the homomorphism $f: \mathbb C[\mathfrak{h}] \rightarrow  \oplus_{d \in \mathbb Z_{\geq 0}}H^0(G/B,\mathcal L_{\alpha_0}^{\otimes d})^T$ is an isomorphism when $G$ is of type $A_n, B_2$ and $C_n$.

By the corollary 3.2, the graded homomorphism $f:\mathbb C[\mathfrak{h}] \rightarrow  \oplus_{d \in \mathbb Z_{\geq 0}}H^0(G/B,\mathcal L_{\alpha_0}^{\otimes d})^T$ is injective.  Hence we have $sym^d(\mathfrak{h}) \subset H^0(G/B , \mathcal L_{\alpha_0}^{\otimes d})^T$.\\
Let $\alpha_0=\sum m_i \varpi_i $. $J:=\{i \in \{ 1,2,\cdots n\};m_i\geq 1\}$.\\
Let $P=P_J$.
Let $U^-_P$ be unipotent radical of the  opposite parabolic subgroup of $P$ determind by $T$ and $B$.\\
Take line bundle $\mathcal L_{\alpha_0}^{\otimes d} $ on $G/P$ and restric to $U^-_P$.\\
Since $U^-_P$ is affine space, $\mathcal L_{\alpha_0}^{\otimes d} $ is trivial on $U^-_P$.\\
So, we have $H^0(U^-_P , \mathcal L_{\alpha_0}^{\otimes d})=\mathbb C [U^-_P]$ , regular fuctions on $U^-_P$.\\
So,  $H^0(G/P , \mathcal L_{\alpha_0}^{\otimes d})$ is a $T$ sub module of $\mathbb C [U^-_P]$. \\
Now considering weights, \\the weight zero in    $H^0(G/P , \mathcal L_{\alpha_0}^{\otimes d})$ corresponding to weight $d \alpha_0$ in $\mathbb C [U^-_P]$. \\
 $X^{a_0}_{-\alpha_0}X^{a_1}_{-\beta_1}X^{a_2}_{-\beta_2} \cdots X^{a_{n-1}}_{-\beta_{n-1}}X^{b_1}_{-\theta_1} \cdots X^{b_{n-1}}_{-\theta_{n-1}}$
 has weight $d \alpha_0$ in $ \mathbb C [U^-_P] $  if and only if \\ $a_0\alpha_0 +\sum_{j=1}^{n-1} a_j\beta_j + \sum_{k=1}^{n-1} b_k\theta_k -d\alpha_0=0 $\\
$\Rightarrow (a_0-d)\alpha_0+\sum_{j=1}^{n-1}a_j\beta_j+\sum_{k=1}^{n-1}b_k\theta_k=0$ \\
$\Rightarrow(a_0-d)\alpha_0+\sum_{j=1}^{n-1}a_j\beta_j+\sum_{k=1}^{n-1}b_k(\alpha_0-\beta_k)=0$ \\
$\Rightarrow(a_0-d)\alpha_0+(\sum_{k=1}^{n-1}b_k)\alpha_0+\sum_{j=1}^{n-1}a_j\beta_j-\sum_{k=1}^{n-1}b_k\beta_k=0$. \\
$\Rightarrow((a_0-d)+\sum_{k=1}^{n-1}b_k)\alpha_0+\sum_{j=1}^{n-1}(a_j-b_j)\beta_j=0$.

Since $\{ \alpha_0 , \beta_j, j=1,2,\cdots,n-1\}$ linearly independent,  we have 
 $a_j=b_j$ for all $j=1,2, \cdots ,n-1$.\\
Since $a_0$ and $b_i's$ are non-negative integers, we have either $a_0=d$ ,
$\sum_{k=1}^{n-1}b_k=0$ or $a_0=0$ and $\sum_{k=1}^{n-1}b_k=d$.\\
Let $V_d:=\{ X^{a_0}_{-\alpha_0}(X_{-\beta_1}X_{-\theta_1})^{a_1}
(X_{-\beta_2}X_{\theta_2})^{a_2} \cdots (X_{-\beta_{n-1}}X_{-\theta_{n-1}}) ^{a_{n-1}}
;\sum_{i=0}^{n-1} a_i=d\}$.\\
In type $A_n , B_2$ and $C_n$ , $dim(G/P)=2n-1$.\\
So, we can identify $H^0(G/P , \mathcal L_{\alpha_0}^{\otimes d})^T$ with $V_d$.   \\
And also we can  idetify  $V_d$ as a subset of $sym^d(h)$.\\
Therefore we have $H^0(G/P , \mathcal L_{\alpha_0}^{\otimes d})^T \subset sym^d(\mathfrak {h})$.\\
So, $H^0(G/P , \mathcal L_{\alpha_0}^{\otimes d})^T=sym^d(\mathfrak{h})$.\\
 Hence the homomorphism  $f: \mathbb C[\mathfrak{h}] \rightarrow  \oplus_{d \in \mathbb Z_{\geq 0}}H^0(G/P,\mathcal L_{\alpha_0}^{\otimes d})^T$ is an isomorphism.\\
Since $H^0(G/B , \mathcal L_{\alpha_0}^{\otimes d}) = H^0(G/P , \mathcal L_{\alpha_0}^{\otimes d})$ ,
the homomorphism $f: \mathbb C[\mathfrak{h}] \rightarrow  \oplus_{d \in \mathbb Z_{\geq 0}}H^0(G/B,\mathcal L_{\alpha_0}^{\otimes d})^T$ is an isomorphism. 
 \\
 
 {\it When $G$ is not of type $A_n,B_2$ or $C_n$, we prove that $dim(G/P)\geq 2n$.}
 
 \underline{Type $B_n, n\neq 2$ :}\\
 In this case, the highest long root $\alpha_0$ is $\varpi_2$.
 So we have $P=P_2$.\\
 The dimension of $U^-_{P_2}=\#\{\alpha \in R^+ / \alpha \geq \alpha_2\} = 4n-5$.\\
 Since $U^-_{P_2}$ affine open subset of $G/P$, $dim(G/P)=4n-5$.\\
 Hence we have $dim(G/P)\geq 2n$ for $n\neq 2$.
 
 \underline{Type $D_n$:}\\
 In this case, the highest long root $\alpha_0$ is $\varpi_2$.
 So we have $P=P_2$.\\
 The dimension of $U^-_{P_2}=\#\{\alpha \in R^+ / \alpha \geq \alpha_2\} = 4n-7$.\\
 Since $U^-_{P_2}$ affine open subset of $G/P$, $dim(G/P)=4n-7$.\\
 Hence we have $dim(G/P)\geq 2n$ for $n\geq 4$.
 
 \underline{Type $E_6$:}\\
 The highest long root $\alpha_0=\varpi_2$. Hence  we have $P=P_2$.\\
 The dimension of $U^-_{P_2}=\#\{\alpha \in R^+ / \alpha \geq \alpha_2\} = 21$.\\
 Then $dim(G/P)=21$.\\
 Hence we have $dim(G/P)>12$.
 
 \underline{Type $E_7$:}\\
 The highest long root $\alpha_0=\varpi_1$. Hence  we have $P=P_1$.\\
 The dimension of $U^-_{P_1}=\#\{\alpha \in R^+ / \alpha \geq \alpha_1\} = 33$.\\
 Then $dim(G/P)=33$.\\
 Hence we have $dim(G/P)>14$.
 
 \underline{Type $E_8$:}\\
 The highest long root $\alpha_0=\varpi_8$. Hence  we have $P=P_8$.\\
 The dimension of $U^-_{P_8}=\#\{\alpha \in R^+ / \alpha \geq \alpha_8\} =57 $.\\
 Then $dim(G/P)=57$.\\
 Hence we have $dim(G/P)>16$.
 
 \underline{Type $F_4$:}\\
 The highest long root $\alpha_0=\varpi_1$. Hence  we have $P=P_1$.\\
 The dimension of $U^-_{P_1}=\#\{\alpha \in R^+ / \alpha \geq \alpha_1\} \geq 8$.\\
 Hence we have $dim(G/P)>8$.
 
 \underline{Type $G_2$:}\\
 The highest long root $\alpha_0=\varpi_2$. Hence  we have $P=P_2$.\\
 The dimension of $U^-_{P_2}=\#\{\alpha \in R^+ / \alpha \geq \alpha_2\} = 5$.\\
 Hence we have $dim(G/P)>4$.
 
  Since the $dim(G/P)\geq 2n$, the krull dimension of $\oplus_{d \in \mathbb Z_{\geq 0}}H^0(G/P,\mathcal L_{\alpha_0}^{\otimes d})>2n$.\\
  Hence  $dim(\oplus_{d \in \mathbb Z_{\geq 0}}H^0(G/P,\mathcal L_{\alpha_0}^{\otimes d})^T)>n$.\\
Therefore the homomorphism $f: \mathbb C[\mathfrak{h}] \rightarrow  \oplus_{d \in \mathbb Z_{\geq 0}}H^0(G/B,\mathcal L_{\alpha_0}^{\otimes d})^T$ is not an isomorphism
 if $G$ is not of the type $A_n, B_2$ or $C_n$.
 
 This completes the proof of the theorem.
  \end{proof}

Let $\mathbb P(\mathfrak g)$ be the projective space corresponding to the affine space $\mathfrak g$.

\begin{corollary} $\mathbb P(\mathfrak g)//G \simeq (G/B(\mathcal L_{\alpha_0}))^{ss}//N_G(T)$, when $G$ is of type $A_n$, $B_2$ or $C_n$. \\
\end{corollary}
\begin{proof} 
By the Chevalley restriction theorem we have \\
the restriction map $\mathbb C[\mathfrak g]^G \longrightarrow \mathbb C[\mathfrak h]^W$ is an isomorphism\\
So we have $\mathbb P(\mathfrak g)//G = \mathbb P(\mathfrak h)//W$\\
Since $G$ is of type $A_n$, $B_2$ or $C_n$, by the theorem 3.3 we have $\mathbb C[\mathfrak{h}] \simeq \oplus_{d \in \mathbb Z_{\geq 0}}H^0(G/B,\mathcal L_{\alpha_0}^{\otimes d})^T$ \\
then $\mathbb C[\mathfrak h]^W \simeq \oplus_{d \in \mathbb Z_{\geq 0}}H^0(G/B,\mathcal L_{\alpha_0}^{\otimes d})^{N_G(T)}$\\
Hence $\mathbb P(\mathfrak h)//G \simeq (G/B(\mathcal L_{\alpha_0}))^{ss}//N_G(T)$.\\
Therefore $\mathbb P(\mathfrak g)//G \simeq (G/B(\mathcal L_{\alpha_0}))^{ss}//N_G(T)$. 
\end{proof}

\section{A description of line bundles $\mathcal L_\chi$ on $G/B$ for which $\oplus_{d \in \mathbb Z_{\geq 0}}H^0(G/B, \mathcal L_\chi^{\otimes d})^T$ is a 
polynomial ring}

In this section, We prove that for any indecomposable dominant character$\chi$ of a maximal torus $T$ of a simple adjoint group $G$ such that there is a Coxeter element $w \in W$ for which $X(w)^{ss}_T(\mathcal L_\chi) \neq \emptyset$,  the graded algebra 
$\oplus_{d \in \mathbb Z_{\geq 0}}H^0(G/B, \mathcal L_\chi^{\otimes d})^T$ is a 
polynomial ring if and only if  $dim(H^0(G/B , \mathcal L_{\chi})^T) \leq$ rank of $G$.

Notation: we use additive notation for the group $X(T)$ of charecters of $T$.

$X(T)^+$ denotes the set of all dominant characters of $T$.

\begin{definition}
 A non trivial dominant character $\chi$ of $T$ is said to be decomposable if there is a pair of non trivial dominant characters $\chi_1$, $\chi_2$ of $T$ such that $\chi=\chi_1+\chi_2$. Otherwise we will call it indecomposable.
\end{definition}

$X(T)^{+}_i$ denotes the set of all indecomposable elements of $X(T)^+$.
\begin{lemma}
 Let $G$ be a simple adjoint group of type $A_{n-1}$. Let $\chi= \sum_{i=1}^{n-1} a_{i} \alpha_i$, 
where $a_i \in \mathbb N $ for each $i=1,2,...,n-1$  be an element of $X(T)_i^+$ such that $\langle \chi,\check{\alpha}_{n-1} \rangle = 0$. Suppose that $X(s_{n-1}...s_1)^{ss}_T(\mathcal L_\chi) \neq \emptyset $ then \\
(i) The coefficients $ a_i$ ,$  i=1,2,...,n-1 $  satisfy the following inequality :\\
\begin{center}
$a_1 \geqslant a_2 \geqslant a_3 \geqslant ....\geqslant a_{n-2}=2$ and $ a_{n-1} =1$.
\end{center}
(ii)  $\chi$ must be of the form $i\varpi_1+\varpi_{n-i}$ for some $2\leq i \leq n-1$.
\end{lemma}
\begin{proof} Since $X(s_{n-1}...s_1)^{ss}_T(\mathcal L_\chi) \neq \emptyset$, we have $s_{n-1}...s_1(\chi)\leq 0$. and also given $\chi$ dominant So, $a_i \geqslant a_{i+1}$ for each $i=1,2,...n-2$.\\
 Now we prove that $a_{n-1}=1$.\\
If $a_{n-1} \geqslant 2 $,  let $i$ be the largest  positive integer such that $a_{n-i}=i a_{n-1}$.\\ Since $2a_{n-1}-a_{n-2}=\langle \chi , \check{\alpha}_{n-1} \rangle = 0 $ , we must have $i \geqslant 2 $. So $a_{n-(i+1)} \neq (i+1)a_{n-1}$\\
$\Longrightarrow a_{n-(i+1)} = ia_{n-1}+c$ , where $0 \leq c \leq a_{n-1}-1$.\\
\underline {case 1}:  If $c = 0$. \\ since $\chi$ dominant and $\langle \chi,\check{\alpha}_{n-1} \rangle = 0$  we have  $a_{n-j}\leq ja_{n-1} ,j=1,2,\cdots n-1$ and also we have  $a_j\geq a_{j+1}$ for each $j=1,2,..n-2$. \\
\underline {claim}: $\chi$ must be of the form $ia_{n-1}(\sum_{j=1}^{n-i}\alpha_j)+ a_{n-1}(\sum_{j=n+1-i}^{n-1}(n-j)\alpha_j)$. \\
Since $c=0$ , $a_{n-(i+1)}=ia_{n-1}$.\\
Now we prove $a_{n-(i+2)}=ia_{n-1}$.\\
Since $\chi$ is dominant, $2a_{n-(i+1)}-a_{n-(i+2)}-a_{n-i} \geq 0$\\
$\Longrightarrow 2ia_{n-1}-a_{n-(i+2)}-ia_{n-i} \geq 0$ \\
$\Longrightarrow ia_{n-i} \geq a_{n-(i+2)}$.\\
and also we have $a_{n-(i+2)} \geq a_{n-(i+1)} = ia_{n-i}$\\
So, $a_{n-(i+2)} = ia_{n-i}$.\\
Similarly, we can prove $a_{n-j}=ia_{n-1}$ for $j=i+3, \cdots, n-1$.\\
Now we prove $a_{n-(i-1)}=(i-1)a_{n-1}$.\\
Since $\chi$ is dominant , $a_{n-(i-1)}+a_{n-3} \geq a_{n-2}+a_{n-i}=2a_{n-1}+ia_{n-1}=(i+2)a_{n-1}$.\\
Since $a_{n-j}\leq ja_{n-1}$ , $a_{n-(i-1)}+a_{n-3} \leq (i+2)a_{n-1}$. \\
So, we conclude that $a_{n-(i-1)}=(i-1)a_{n-1}$ and $a_{n-3}=3a_{n-1}$.\\
Similarly, we can prove that $a_{n-j}=ja_{n-1}$ for $j=i-2,i-3,\cdots , 4$.\\
$\chi$ is of the form $ia_{n-1}(\sum_{j=1}^{n-i}\alpha_j)+ a_{n-1}(\sum_{j=n+1-i}^{n-1}(n-j)\alpha_j)=a_{n-1}(i \varpi_1+\varpi_{n-i})$.\\
This forces that $\chi$ is decomposable, since $a_{n-1}\geq 2$. This is a contradiction to the indecomposability of $\chi$. \\
This proves that if $c=0$ then $a_{n-1}=1$ \\
\underline{case 2}: If $c>0$.\\
 $\langle \chi , \check{\alpha}_1 \rangle = 2a_1-a_2 \geq a_1 \geq a_{n-(i+1)}=ia_{n-1}+c$.\\
 Similarly, $\langle \chi, \check{\alpha}_{n-i} \rangle \geq 2ia_{n-1}-ia_{n-1}-c-(i-1)a_{n-1} =a_{n-1}-c \geq 1 $.

Thus, $\chi -(i\varpi_1+\varpi_{n-i})$ is still a non zero dominant weight which is in the root lattice. 
This is a contradiction to the indecomposability of $\chi$. \\
So, $c=0$ and $a_{n-1}=1$ and this proves (i).\\
(ii). Using above arguement we also see that $\chi$ is of the form $i\varpi_1+\varpi_{n-i}$ where $i$ is the largest positive integer such that $a_{n-i}=i a_{n-1}$.

\end{proof}
\begin{lemma}

 Let $G$ be a simple adjoint group of type $A_{n-1}$. Let $\chi=\sum_{i=1}^{n-1}a_i\alpha_i$ , where $a_i \in \mathbb N,i=1 \cdots,n-1$ be an element of $X(T)_i^+$ such that $\langle \chi , \check{\alpha}_1 \rangle = 0$.Suppose that $X(s_1\cdots s_{n-1})_T^{ss}(\mathcal L_\chi) \neq \emptyset$, then \\
(i) The coefficient $a_i , i=1 ,\cdots, n-1$ satisfy the following inequality:\\
$$1=a_1\leq a_2\leq \cdots \leq a_{n-1}$$.\\
(ii) $\chi$ must be of the form $\varpi_i+i\varpi_{n-1}$ for some $2\leq i \leq n-1$.
\end{lemma}
\begin{proof}
 Similar to lemma 4.2.
\end{proof}

\begin{lemma}
 Let G be a simple adjoint group of type $A_{n-1}$. Let $\chi=\sum_{i=1}^{n-1}a_i\alpha_i$ , where $a_i \in \mathbb N,i=1 \cdots,n-1$ be an element of $X(T)_i^+$.
If $X(s_{i+1}...s_{n-1}s_i...s_1)_T^{ss}(\mathcal L_\chi)\neq \emptyset$ for some $2 \leq i \leq n-3$, then $\chi=\alpha_1+\cdots+\alpha_{n-1}$.
\end{lemma}

\begin{proof} Since  $X(s_{i+1}...s_{n-1}s_i...s_1)_T^{ss}(\mathcal L_\chi)\neq \emptyset$ for some $2 \leq i \leq n-3$, $s_{i+1}...s_{n-1}s_i...s_1(\chi) \leq 0 $.\\ So, we have
 $\sum_{j=1}^i(a_{j+1}-a_j)\alpha_j+(a_{i+1}-a_1-a_{n-i})\alpha_{i+1}+\sum_{k=i+2}^{n-1}(a_{i+2}-a_{n-1})\alpha_k \leq 0$.\\  Since $\chi$ is dominant, we have $a_{i+1}\leq a_i \leq \cdots \leq a_2 \leq a_1$,
$ a_{i+1} \leq a_{i+2} \leq \cdots \leq a_{n-1}$
and  $2a_{i+1}-a_i-a_{i+1} \geq 0$ then $a_{i+1}=a_i=a_{i+2}$.\\
Similarly, we can prove that $a_1=a_2= \cdots = a_{n-1}$.\\
Therefore  $\chi =a_1(\alpha_1+\alpha_2+ \cdots + \alpha_{n-1})$.\\
Since $\chi$ is indecomposable and $a_i \in \mathbb N$, we have $a_1=1$.\\
Hence $\chi=\alpha_1+\alpha_2+ \cdots + \alpha_{n-1}$.

\end{proof}
\begin{lemma}
 Let $G$ be a simple adjoint group of type $A_{n-1}$. Let $\chi=i\varpi_1+\varpi_{n-i} \in X(T)_i^+$ for some $2\leq i \leq n-3
 $ then   $dim(H^0($G/B$,\mathcal L_\chi)^T) > n-1$.
\end{lemma}
\begin{proof}
 Since $\chi= i\varpi_i+\varpi_{n-i}$, each integer in $\{1,2,\cdots,n\}$ occurs in the standard
 young tableau corresponding to $T$ - invariant standard monomial of shape $\chi$ (refer to \cite{r11} for standard monomial)
 exactly once.\\
 Hence  $dim(H^0($G/B$,\mathcal L_\chi)^T)= {n-1\choose i}$.\\
 Since $2\leq i \leq n-3$, $i < n-2$ and so $i-j < n-(j+2)$ for $j=1,2,\cdots i-2 $ .\\
 So we have $(n-2)(n-3)\cdots(n-i) > i !$.\\
 Hence ${n-1 \choose i} > n-1$.\\
 Therefore  $dim(H^0($G/B$,\mathcal L_\chi)^T) > n-1$.
 
\end{proof}

\begin{lemma}
 Let $G$ be a simple adjoint group of type $A_{n-1}$. Let $\chi=\varpi_i+i\varpi_{n-1} \in X(T)_i^+$ for some $2\leq i \leq n-3
 $ then   $dim(H^0($G/B$,\mathcal L_\chi)^T) > n-1$.
\end{lemma}
\begin{proof}
 Let $w_0$ be  longest weyl group element. \\
 Since $\chi =\varpi_i+i\varpi_{n-1}\in X(T)_i^+, -w_0 \chi =  i\varpi_1+\varpi_{n-i}\in X(T)_i^+$.\\
 Since  $H^0($G/B$,\mathcal L_{\chi})^*=H^0($G/B$,\mathcal L_{-w_0\chi})$ , $dim(H^0($G/B$,\mathcal L_{\chi})^T)=dim(H^0($G/B$,\mathcal L_{-w_0\chi})^T)$.\\
 By the previous  lemma we have $dim(H^0($G/B$,\mathcal L_{-w_0\chi})^T) > n-1$.\\
 Hence  $dim(H^0($G/B$,\mathcal L_\chi)^T) > n-1$.
 
\end{proof}

\begin{lemma}
Let $G$ be a simple adjoint group of type $A_{n-1}$, $n \neq 4$.
 Let $\chi \in X(T)^+$. If $w \in W$ is a Coxeter element such that $X(w)_T^{ss}(\mathcal L_\chi) \neq \emptyset$
then $\{ i\in \{1,\cdots,n\}:l(ws_i)=l(w)-1\}\subseteq\{1,n-1\}$.
\end{lemma}
\begin{proof}
Let $\chi= \sum_{i=1}^{n-1}a_i\alpha_i$, where $a_i \in \mathbb {N}$.
Suppose there is a $2\leq i \leq n-2$ such that $l(ws_i)=l(w)-1$. Since $w\chi\leq 0$ we have $a_{i-1}+a_{i+1}\leq a_i$. Since $\langle \chi , \check{\alpha}_{i-1} \rangle \geq 0$ and $\langle \chi , \check{\alpha}_{i+1} \rangle \geq 0$ we have $2a_{i-1} \geq a_{i-2}+a_i$ and 
$2a_{i+1}\geq a_i+a_{i+2}$.\\
So, we have $2a_i \geq  2(a_{i-1}+a_{i+1})\geq2a_i+a_{i-2}+a_{i+2}$\\ $\Longrightarrow a_{i-2}+a_{i+2}=0 $\\ $\Longrightarrow a_{i-2}=a_{i+2}=0 $\\ $\Longrightarrow i-2 \leq 0$ and $i+2 \geq n $\\ $\Longrightarrow i=2 $ and $i=n-2 $\\ $\Longrightarrow i=2$ and $n=4$ which is contradiction to assumption $n\neq 4$. This completes the proof of the lemma.
\end{proof}

We now prove the following theorem.\\

\begin{theorem}
Let $G$ be a simple adjoint group over $\mathbb C$. Let $\chi \in X(T)^{+}_i$ be such that 
there is a Coxeter element $w \in W$ for which $X(w)^{ss}_T(\mathcal L_\chi) \neq \emptyset$, 
 the graded algebra 
$\oplus_{d \in \mathbb Z_{\geq 0}}H^0(G/B, \mathcal L_\chi^{\otimes d})^T$ is a 
polynomial ring if and only if  $dim(H^0(G/B , \mathcal L_{\chi})^T) \leq$ rank of $G$ .

\end{theorem}

\begin{proof}
 We prove the theorem by using case by case analysis. 

For a given simple adjoint group $G$ and for any indecomposable dominant character $\chi$ of $T$ such that 
$X(w)^{ss}_T(\mathcal L_\chi) \neq \emptyset$ for some Coxeter element $w$ in $W$, we prove that either the graded algebra 
$\oplus_{d \in \mathbb Z_{\geq 0}}H^0(G/B, \mathcal L_\chi^{\otimes d})^T$ is a 
polynomial ring and  $dim(H^0(G/B , \mathcal L_{\chi})^T) \leq$ rank of $G$  or the graded algebra \\
$\oplus_{d \in \mathbb Z_{\geq 0}}H^0(G/B, \mathcal L_\chi^{\otimes d})^T$ is a not
polynomial ring and  $dim(H^0(G/B , \mathcal L_{\chi})^T) >$ rank of $G$ .

 \underline{Type $A_{n-1}$, $n\neq 4$}: 
 
 From lemma 4.7, if $w \in W$ a Coxeter elements $w$ such that $X(w)_{T}^{ss}(L_{\chi})$ is non empty  then 
$w=s_{i+1}.\cdots s_{n-1}s_{i}.\cdots s_1$ for some $1\leq i \leq n-2$ or $w=s_1.\cdots s_{n-1}$.

When $w=s_{n-1}.\cdots s_1$ using lemma 4.2, the indecomposable dominant character $\chi$ of $T$ for which 
 $X(w)_{T}^{ss}(L_{\chi})$ is non empty are $\chi=i\varpi_1+\varpi_{n-i}$ for 
$1\leq i \leq n-1$.

\underline{If i=n-1}: $\chi=n\varpi_1$ , in this case there is only one $T$ invariant monomial. \\
Hence $\oplus_{d \in \mathbb Z_{\geq 0}}H^0(G/B, \mathcal L_\chi^{\otimes d})^T$ is a 
polynomial ring in one variable.\\
\underline{If i=n-2}:\\
we have $\chi=(n-2)\varpi_1+\varpi_2$. \\
$dim(H^0($G/B$,\mathcal L_\chi)^T)=n-1$. \\
Consider the map $\phi : \mathbb C[X_1,\cdots,X_{n-1}] \longrightarrow \oplus_{d \in \mathbb Z_{\geq 0}}H^0(G/B, \mathcal L_\chi^{\otimes d})^T$ 
is given by \\$X_i \longmapsto p_{1i}p_{23\cdots \hat{i} \cdots n}$ where $p_{1i}, p_{23\cdots \hat{i} \cdots n}$ Pl\"{u}ker co-ordinates and $\hat{i}$ denotes omiting of $i$ .\\ 
 Using the standard monomial of shape $d\chi$ we can see that $\phi$ is surjective.\\
So we have 

(1) the surjective map $\phi: \mathbb C[X_1,\cdots,X_{n-1}] \longrightarrow \oplus_{d \in \mathbb Z_{\geq 0}}H^0(G/B, \mathcal L_\chi^{\otimes d})^T$.\\

Let $P=P_1 \cap P_2$. 
Since $\oplus_{d \in \mathbb Z_{\geq 0}}H^0(G/B, \mathcal L_\chi^{\otimes d})^T =\oplus_{d \in \mathbb Z_{\geq 0}}H^0(G/P, \mathcal L_\chi^{\otimes d})^T$, we have \\ 

(2) the krull dimension of $\oplus_{d \in \mathbb Z_{\geq 0}}H^0(G/B, \mathcal L_\chi^{\otimes d})^T = n-1$.\\

From (1) and (2) we conclude that the map $ \mathbb C[X_1,\cdots,X_{n-1}] \longrightarrow \oplus_{d \in \mathbb Z_{\geq 0}}H^0(G/B, \mathcal L_\chi^{\otimes d})^T$
 is an isomorphism.\\
 Hence $\oplus_{d \in \mathbb Z_{\geq 0}}H^0(G/B, \mathcal L_\chi^{\otimes d})^T $
is a polynomial ring.\\
\underline{If $2\leq i \leq n-3$}:\\
$\chi=i\varpi_1+\varpi_{n-i}$, by the lemma 4.5 we have 
$dim(H^0($G/B$,\mathcal L_\chi)^T) > n-1$.\\
\underline{claim}: $\oplus_{d \in \mathbb Z_{\geq 0}}H^0(G/B, \mathcal L_\chi^{\otimes d})^T$
is not a polynomial ring.\\
Let $P=P_1 \cap P_{n-i}$. 
Since $\oplus_{d \in \mathbb Z_{\geq 0}}H^0(G/B, \mathcal L_\chi^{\otimes d})^T =\oplus_{d \in \mathbb Z_{\geq 0}}H^0(G/P, \mathcal L_\chi^{\otimes d})^T$, we have the krull dimension of $\oplus_{d \in \mathbb Z_{\geq 0}}H^0(G/B, \mathcal L_\chi^{\otimes d})^T = 1+i(n-1-i)$.
\\ On the other hand we have $dim(H^0($G/B$,\mathcal L_\chi)^T)={n\choose i}$. \\
 With out loss of generality, we may assume that $i\leq (n-1)/2$.\\
So, we have $i(n-1-i)<(n-1)(n-i)/2 $.\\
 Since $i\leq n-3 , i-j<n-(j+2)$ for $j=1,\cdots,i-3$.\\
 So $(n-2)(n-3)\cdots (n-i+1)>3.4.\cdots.i$\\
 $\Longrightarrow {n\choose i}> 1+i(n-1-i)$.\\
 Then we have $ dim(H^0($G/B$,\mathcal L_\chi)^T)$ $>$ krull dimension of $\oplus_{d \in \mathbb Z_{\geq 0}}H^0(G/B, \mathcal L_\chi^{\otimes d})^T $.\\
  Hence 
$\oplus_{d \in \mathbb Z_{\geq 0}}H^0(G/B, \mathcal L_\chi^{\otimes d})^T$
is not a polynomial ring.\\
 \underline{If $i=1$}:\\
we have $\chi=\varpi_1+\varpi_{n-1}=\alpha_1+\cdots +\alpha_{n-1}$. By the theorem 3.3, we have 
the ring 
 $\bigoplus_{d\in \mathbb Z_{\geq 0}}H^{0}(G/B , \mathcal L_{\alpha_1+\cdots +\alpha_{n-1}}^{\otimes d})^T$ is a polynomial ring.

When $w=s_{1}.\cdots s_{n-1}$, By the lemma 2.7, the indecomposable dominant character $\chi$ of $T$ for which $X(w)_{T}^{ss}(L_{\chi})$ is non empty are\\
$\chi=\varpi_i+i\varpi_{n-1}$ for $1\leq i \leq n-1$.\\
\underline{If $i=n-1$}:\\
we have $\chi=n\varpi_{n-1}$.\\ 
Since $-w_0\chi=n\varpi_1, \oplus_{d \in \mathbb Z_{\geq 0}}H^0(G/B, \mathcal L_\chi^{\otimes d})^T$ and
 $\oplus_{d \in \mathbb Z_{\geq 0}}H^0(G/B, \mathcal L_{n\varpi_1}^{\otimes d})^T$ are isomorphic.
 We proved that $\oplus_{d \in \mathbb Z_{\geq 0}}H^0(G/B, \mathcal L_{n\varpi_1}^{\otimes d})^T$ is a polynomial ring.\\
 Hence $\oplus_{d \in \mathbb Z_{\geq 0}}H^0(G/B, \mathcal L_\chi^{\otimes d})^T$ is a polynomial ring.\\
\underline{If i=n-2}: \\
In this case $\chi=\varpi_{n-2}+(n-2)\varpi_{n-1}$.\\
Since $-w_0\chi=\varpi_2+(n-2)\varpi_1, \oplus_{d \in \mathbb Z_{\geq 0}}H^0(G/B, \mathcal L_\chi^{\otimes d})^T$ and
 $\oplus_{d \in \mathbb Z_{\geq 0}}H^0(G/B, \mathcal L_{\varpi_2+(n-2)\varpi_1}^{\otimes d})^T$ are isomorphic.\\
 We proved that $\oplus_{d \in \mathbb Z_{\geq 0}}H^0(G/B, \mathcal L_{\varpi_2+(n-2)\varpi_1}^{\otimes d})^T$
 is a polynomial ring. \\
 Hence $\oplus_{d \in \mathbb Z_{\geq 0}}H^0(G/B, \mathcal L_\chi^{\otimes d})^T$ is a polynomial ring.\\
 \underline{If $2\leq i \leq n-3$}:\\
 $\chi=\varpi_i+i\varpi_{n-1}$, by the lemma 4.6, we have 
$dim(H^0($G/B$,\mathcal L_\chi)^T) > n-1$.\\
\underline{claim}: $\oplus_{d \in \mathbb Z_{\geq 0}}H^0(G/B, \mathcal L_\chi^{\otimes d})^T$
is not a polynomial ring.\\
 Since $-w_0\chi=\varpi_{n-i}+i\varpi_1, \oplus_{d \in \mathbb Z_{\geq 0}}H^0(G/B, \mathcal L_\chi^{\otimes d})^T$ and
 $\oplus_{d \in \mathbb Z_{\geq 0}}H^0(G/B, \mathcal L_{\varpi_{n-i}+i\varpi_1}^{\otimes d})^T$ are isomorphic.\\
 We proved that $\oplus_{d \in \mathbb Z_{\geq 0}}H^0(G/B, \mathcal L_{\varpi_{n-i}+i\varpi_1}^{\otimes d})^T$
 is not a polynomial ring. \\
 Hence $\oplus_{d \in \mathbb Z_{\geq 0}}H^0(G/B, \mathcal L_\chi^{\otimes d})^T$ is not a polynomial ring.
  
\underline{If $i=1$}:\\
we have $\chi=\varpi_1+\varpi_{n-1}=\alpha_1+\cdots +\alpha_{n-1}$. By the theorem 3.3, we have 
the ring 
 $\bigoplus_{d\in \mathbb Z_{\geq 0}}H^{0}(G/B , \mathcal L_{\alpha_1+\cdots +\alpha_{n-1}}^{\otimes d})^T$ is a polynomial ring.\\

When $w=s_{i+1}\cdots s_{n-1}s_i\cdots s_1$ , where $2\leq i\leq n-3$ , By the lemma 4.4,  the indecomposable dominant character $\chi$ such that $X(w)_{T}^{ss}(L_{\chi})$ is non empty is $\alpha_1+\cdots +\alpha_{n-1}$. By the theorem 2.3, the ring  $\bigoplus_{d\in \mathbb Z_{\geq 0}}H^{0}(G/B , \mathcal L_{\alpha_1+\cdots +\alpha_{n-1}}^{\otimes d})^T$ is a polynomial ring.

\underline{Type $A_{3}$}:

The indecomposable dominant characters for which there is a coxeter $w$ such that $X(w)_{T}^{ss}(L_{\chi})$ is non empty
are $\alpha_{1}+\alpha_{2}+\alpha_{3}$, $3\alpha_{1}+2\alpha_{2}+\alpha_{3}$,
$\alpha_{1}+2\alpha_{2}+\alpha_{3}$ and 
$\alpha_{1}+2\alpha_{2}+3\alpha_{3}$.

When $\chi=\alpha_1+\alpha_2+\alpha_3=\alpha_0$, by theorem 3.3, we have  $\bigoplus_{d\in \mathbb Z_{\geq 0}}H^{0}(G/B , \mathcal L_{\chi}^{\otimes d})^T$ is a polynomial ring .

When $\chi=3\alpha_{1}+2\alpha_{2}+\alpha_{3}=4\varpi_1$, there is only one $T$ invariant monomial. \\
Hence $\oplus_{d \in \mathbb Z_{\geq 0}}H^0(G/B, \mathcal L_\chi^{\otimes d})^T$ is a 
polynomial ring in one variable.

When $\chi=\alpha_{1}+2\alpha_{2}+3\alpha_{3}=4\varpi_3$.
Since $-w_0\chi=4\varpi_1, \oplus_{d \in \mathbb Z_{\geq 0}}H^0(G/B, \mathcal L_\chi^{\otimes d})^T$ and
 $\oplus_{d \in \mathbb Z_{\geq 0}}H^0(G/B, \mathcal L_{4\varpi_1}^{\otimes d})^T$ are isomorphic.\\
  Hence $\oplus_{d \in \mathbb Z_{\geq 0}}H^0(G/B, \mathcal L_\chi^{\otimes d})^T$ is a polynomial ring. 

Now we deal the special case of $\chi=2\varpi_{2}$ in $A_{3}$.
In this case the Coxeter element is $w=s_1s_3s_2$. 
Let a typical monomial  
$\prod_{i <j}p_{ij}^{m_{ij}}$ in the Pl\"{u}ker co-ordinates which is $T$-invariant. 
Then it is easy to see that each of the indices $1,2,3,4$ occur same number of 
times. So if $p_{12}$ (resp. $p_{13}$) is a factor of $T$ invariant monomial $M$, then $p_{34}$ (resp. $p_{24}$) 
is also a factor of $M$. Also, if $p_{14}$ is a factor of $T$ invariant monomial $M$, then $p_{23}$ also a factor of $M$. But by the Pl\"{u}ker relation we have 
\[p_{14}p_{23}=p_{13}p_{24}-p_{12}p_{34}.\] So, $p_{13}p_{24}$ and $p_{12}p_{34}$
generate the ring of $T$ invariants of $\bigoplus_{d\in \mathbb Z_{\geq 0}}H^{0}(G/P_2,\mathcal L_\chi^{\otimes d })$.
where $P_2$ is the the maximal parabolic subgroup associated to $\alpha_2$.\\

(1) Hence we have a surjective map $\mathbb C[p_{13}p_{24},p_{12}p_{34}] \longrightarrow \bigoplus_{d\in \mathbb Z_{\geq 0}}H^{0}(G/P_2,\mathcal L_\chi^{\otimes d })^T$.\\

(2) The krull dimension of $\bigoplus_{d\in \mathbb Z_{\geq 0}}H^{0}(G/P_2,\mathcal L_\chi^{\otimes d })^T$ is two.\\

From (1) and (2) we conclude that the map $\mathbb C[p_{13}p_{24},p_{12}p_{34}] \longrightarrow \bigoplus_{d\in \mathbb Z_{\geq 0}}H^{0}(G/P_2,\mathcal L_\chi^{\otimes d })^T$ 
is an isomorphism.\\ 
So, $\oplus_{d \in \mathbb Z_{\geq 0}}H^0(G/P_2, \mathcal L_{2\varpi_2}^{\otimes d})^T$ is a polynomial algebra. \\
Hence $\oplus_{d \in \mathbb Z_{\geq 0}}H^0(G/B, \mathcal L_{2\varpi_2}^{\otimes d})^T$ is a polynomial algebra.\\
\\
This completes the proof for the type $A_{n-1}$.\\

\underline{Type $B_n,n\neq 2$}:

By the theorem 4.2 in \cite{r6}, when $G$ is of type $B_n$, the coxeter elements $w$ for which there is dominant
character such that $X(w)_{T}^{ss}(L_{\chi})$ is non empty is $s_ns_{n-1}\cdots s_2s_1$.
The indicomposable dominant character with this property is $\chi=\alpha_1+\alpha_2+\cdots +\alpha_n=\varpi_1$.\\    

  Now consider the standard representation $\mathbb C^{2n+1}$ of $SO_{2n+1}$. Then
\[(1) \hspace{3cm} dim(Sym^2(\mathbb C^{2n+1})^*))=(n+1)(2n+1).\]

By Weyl dimension formula, the dimension of the irreducible representation $V(2\varpi_1)$ of $SO_{2n+1}$
is 
\[\prod_{\alpha \in \Phi^+}\frac{\langle 2\varpi_1+\rho, \check{\alpha} \rangle}
{\langle \rho , \check{\alpha} \rangle}.\]

Again since $\langle 2\varpi_1+\rho, \check{\alpha} \rangle= \langle \rho , \check{\alpha} \rangle$ for $\alpha \ngeq \alpha_1$, we have 
\[dim(V(2\varpi_1))= \prod_{\alpha \in \Phi^+, \alpha \geq \alpha_1}\frac{\langle 2\varpi_1+\rho, \check{\alpha} \rangle}{\langle \rho , \check{\alpha} \rangle}.\]
The set of $\alpha \in \Phi^+$ such that $\alpha \geq \alpha_1$ is $\{\alpha_1, \alpha_1+\alpha_2, \cdots ,\alpha_1+\alpha_2+\cdots +\alpha_n, \alpha_1+\alpha_2+\cdots +\alpha_{n-1}+2\alpha_n, \cdots, \alpha_1+2(\alpha_2+\cdots +\alpha_n)\}$.\\
We now calculate $\frac{\langle 2\varpi_1+\rho, \check{\alpha} \rangle}{\langle \rho , \check{\alpha} \rangle}$ for all $\alpha \geq \alpha_1$ .\\
\[\frac{\langle 2\varpi_1+\rho, \check{\alpha_1+\cdots +\alpha_i} \rangle}{\langle \rho , \check{\alpha_1+\cdots \alpha_i} \rangle}=\frac{i+2}{i}, \hspace{0.5cm} 1 \leq i \leq n-1.\]

\[\frac{\langle 2\varpi_1+\rho, \check{\alpha_1+\cdots +\alpha_n} \rangle}{\langle \rho , \check{\alpha_1+\cdots \alpha_n} \rangle}=\frac{2n+3}{2n-1}.\]

\[\frac{\langle 2\varpi_1+\rho,\check{\alpha_1+\cdots +\alpha_{i-1}+2(\alpha_i+\cdots \alpha_n)} \rangle}{\langle \rho , \check{\alpha_1+\cdots +\alpha_{j-1}+2(\alpha_j+\cdots \alpha_n)}\rangle}=\frac{2n-j+2}{2n-j}, \hspace{0.5cm} 2 \leq j \leq n.\]\\
Hence 
\[(2) \hspace{3cm} dim(V(2\varpi_1))= \prod_{\alpha \in \Phi^+, \alpha \geq \alpha_1}\frac{\langle 2\varpi_1+\rho, \check{\alpha} \rangle}{\langle \rho , \check{\alpha} \rangle}= n(2n+3).\]

From (1) and (2) we can conclude that $(Sym^2((\mathbb C^{2n+1})^*)^{SO_{2n+1}}$ is one dimensional, namely generated by the quadratic form $q$ which defines the orthogonal group $SO_{2n+1}$. Hence we have 
\[Sym^2(\mathbb C^{2n+1})^*=V(2\varpi_1)^*+\mathbb C q.\] 
where  $q=\sum_{i=1}^nX_iX_{2n+2-i}$.  Since $q$-vanishes on $SO_{2n+1}(\mathbb C)/P_1$, where $P_1$ is the maximal parabolic associated to $\alpha_1$,
there is a unique quadratic  relation among the variables 
$X_iX_{2n+2-i}, \,\, i=1,2,\cdots ,n+1$ on $SO_{2n+1}(\mathbb C)/P_1$ , namely $aX_{n+1}^2=\sum_{i}^nX_iX_{2n+2-i}$ for some non zero $a\in \mathbb C$ on $SO_{2n+1}(\mathbb C)/P_1$ (refer to \cite{r7}).

Now we explain all the $T$-invariant polynomials restricted to $SO_{2n+1}(\mathbb C)/P_1$. Take a $T$-invariant polynomial $X_1^{m_1}X_2^{m_2}\cdots X_{2n+1}^{m_{2n+1}}$ with $m_i=m_{2n+2-i}$. The above relatilon  implies that every $T$-invariant 
polynomial restricted to  $SO_{2n+1}(\mathbb C)/P_1$ is a linear combination of the monomials of the form $(X_1X_{2n+1})^{r_1}(X_2X_{2n})^{r_2}\cdots (X_{n-1}X_{n+3})^{r_{n-1}}(X_nX_{n+2})^{r_n}$ for some $ r_i's$ in $\mathbb Z_{\geq 0}$. Thus 

(3) the map 
$\mathbb C[X_1X_{2n+1}, X_2X_{2n}, \cdots X_{n-1}X_{n+3},X_nX_{n+2}] \rightarrow \oplus_{d \in \mathbb Z_{\geq 0}}H^0(G/P_1, \mathcal L_{2\varpi_1}^{\otimes d})^T$ is onto.  
\\

On the other hand we have $dim(U_{P_1}^-)= |\{\alpha \in R^+: \alpha \geq \alpha_1\}|=2n-1$, where $U^-_{P_1}$ be unipotent radical of the  opposite parabolic subgroup of $P_1$ determind by $T$ and $B$.\\
Since $U_{P_1}^-$ is open subset of $G/P_1$, the dimension of the affine cone over $G/P_1$ is of dimension $2n$.
So we have 

(4) the Krull dimension of $\oplus_{d \in \mathbb Z_{\geq 0}}H^0(G/P_1, \mathcal L_{2\varpi_1}^{\otimes d})^T$ is $2n-n=n$.  
\\

From (3) and (4) we conclude that \\$\oplus_{d \in \mathbb Z_{\geq 0}}H^0(G/P_1, \mathcal L_{\varpi_1}^{\otimes d})^T = \mathbb C[X_1X_{2n+1},X_2X_{2n},\cdots,X_{n-1}X_{n+3},X_nX_{n+2}]$ is a polynomial ring. \\
Hence $\oplus_{d \in \mathbb Z_{\geq 0}}H^0(G/B, \mathcal L_{\varpi_1}^{\otimes d})^T$ is a polynomial ring.
\\

Now we prove that $dim(H^0(G/B , \mathcal L_{2\varpi_1})^T) \leq$ rank of $G$ .\\

 The $T$ invariant monomials in $Sym^2((\mathbb C^{2n+1})^*)$ are of the form $X_iX_j$ for some $ 1\leq i \leq n+1, j=2n+2-i$  .\\
Hence we have $dim(Sym^2((\mathbb C^{2n+1})^*))^T)$ is $n+1$.\\
Since $Sym^2((\mathbb C^{2n+1})^*)=V(2\varpi_1)^*+\mathbb C q$, $dim(V(2\varpi_1)^T)=n$.\\
Hence $dim(H^0(G/B , \mathcal L_{2\varpi_1})^T) =$ rank of $G$ .\\

This completes the proof for type $B_n,n\neq 2$.\\

\underline{Type $B_2$:}

By the theorem 4.2 in \cite{r6}, when $G$ is of type $B_2$, the coxeter elements $w$ for which there is dominant
character such that $X(w)_{T}^{ss}(L_{\chi})$ is non empty are $ s_2s_1$ and $s_1s_2$.
The indicomposable dominant character with this property for the coxeter element $s_2s_1$ is $\chi=\alpha_1+\alpha_2$.\\  
In this case using similar arguements as in type $B_n,n \neq 2$, we can prove that the ring $\oplus_{d \in \mathbb Z_{\geq 0}}H^0(G/B, \mathcal L_{\chi}^{\otimes d})^T$ is a polynomial ring.\\

The indicomposable dominant character $\chi$ for the coxeter element $w=s_1s_2$  for which $X(w)_{T}^{ss}(L_{\chi})$ is non empty is $\chi=\alpha_1+2\alpha_2=\alpha_0$. 
By the theorem 3.3, the ring $\oplus_{d \in \mathbb Z_{\geq 0}}H^0(G/B, \mathcal L_{\alpha_0}^{\otimes d})^T$ is a polynomial ring.\\

\underline{Type $C_n$:}

When $G$ is of the type $C_n$, by theorem 4.2 of \cite{r6}, the only coxeter element $w$ for which there is dominant
character such that $X(w)_{T}^{ss}(L_{\chi})$ is non empty is $w=s_ns_{n-1}\cdots s_2s_1$.\\
Further, the indicomposable dominant character $\chi$ with this property is $2\varpi_1=2(\sum_{i\neq n}\alpha_i)+\alpha_n=\alpha_0$. \\
By the theorem 3.3,  the ring of $T$ invarinats  $\oplus_{d \in \mathbb Z_{\geq 0}}H^0(G/B, \mathcal L_{\varpi_1}^{\otimes d})^T$ is a polynomial ring .
\\

We now prove that $dim(H^0(G/B , \mathcal L_{2\varpi_1})^T) =$ rank of $G$ .\\
Since $\chi=\alpha_0$, $H^0(G/B , \mathcal L_{\chi})=\mathfrak{g}$.\\
So, we have $H^0(G/B , \mathcal L_{\chi})^T=\mathfrak{h}$.\\
Hence $dim(H^0(G/B , \mathcal L_{2\varpi_1})^T) =$ rank of $G$ .\\

\underline{Type $D_4$:} 

By theorem 4.2 of \cite{r6}, the only coxeter elements $w \in W$ for which there exist a dominant weight $\chi$ such that $X(w)_{T}^{ss}(L_{\chi})$ is non empty are $ s_4s_3s_2s_1, s_4s_1s_2s_3$ and $s_3s_1s_2s_4$. \\
The indecomposable dominat characters with this property are $ 2(\alpha_1+\alpha_2)+\alpha_3+\alpha_4$, $2(\alpha_3+\alpha_2)+\alpha_1+\alpha_4$ and
$2(\alpha_4+\alpha_2)+\alpha_1+ \alpha_3$ to the coxeter elements $ s_4s_3s_2s_1, s_4s_1s_2s_3$ and $s_3s_1s_2s_4$  respectively.

Since there is an automorphism of the Dynkin diagram sending $\alpha_1$ to $\alpha_3$ and fixing $\alpha_2$ and $\alpha_4$ and there is also an automorphism that sends $\alpha_1$ to $\alpha_4$ and fixing $\alpha_2$ and $\alpha_3$.\\
If $\sigma'$ is an automorphism of dynkin diagram, we get a $\sigma:G\rightarrow G$ automorphism of algebraic groups such that $\sigma(B)=B$, $\sigma(T)=T$ and $\sigma (\alpha_i)=\sigma '(\alpha_i)$ for all $i=1,\cdots,4.$\\Further, we have 
$H^0(G/B, \mathcal L_\chi)$ and $ H^0(G/B, \mathcal L_{\sigma(\chi)})$ are isomorphic as $G$-modules where the action of $G$ on $ H^0(G/B, \mathcal L_{\sigma(\chi)})$ via $\sigma$.\\
Thus, $ H^0(G/B, \mathcal L_\chi)^T= H^0(G/B, \mathcal L_{\sigma(\chi)})^T$.\\
So it is sufficient to consider the case when $\chi=2\varpi_1=2\alpha_1+2\alpha_2+\alpha_3+\alpha_4$.

Now consider the standard representation $\mathbb C^{8}$ of $SO_{8}$. Then, we have
\[(1) \hspace{3cm} dim(Sym^2((\mathbb C^{8})^*))=36.\] 

By using the Weyl dimension formula and by proceeding with similar calculation above we can see that the dimension of the irreducible representation $V(2\varpi_1)$ is 35. \hspace{1cm} (2)

From (1) and (2) we have  
\[Sym^2((\mathbb C^8)^*)=V(2\varpi_1)^*+\mathbb C q.\] 
where  $q=\sum_{i=1}^8X_iX_{9-i}$.  Since $q$-vanishes on $SO_{8}(\mathbb C)/P_1$, where $P_1$ is the maximal parabolic associated to $\alpha_1$,
there is a unique quadratic  relation among the variables $X_iX_{9-i}, \,\, i=1,2,3,4$ on $SO_{8}(\mathbb C)/P_1$, namely $-X_1X_8=\sum_{i}^4X_iX_{9-i}$ on $SO_{8}(\mathbb C)/P_1$(refer to \cite{r7}).

Now we explain all the $T$-invariant polynomials restricted to $SO_{8}(\mathbb C)/P_1$. Take a $T$-invariant polynomial $X_1^{m_1}X_2^{m_2}\cdots X_{8}^{m_{8}}$ with $m_i=m_{9-i}$. The above relatilon  implies that every $T$-invariant 
polynomial restricted to  $SO_{8}(\mathbb C)/P_1$ is a linear combination of the monomials of the form $(X_2X_{7})^{m_2}(X_3X_6)^{m_3}\cdots (X_4X_5)^{m_4}$. Thus 

(3) the map 
$\mathbb C[X_2X_7, X_3X_6, X_{4}X_{5}] \rightarrow \oplus_{d \in \mathbb Z_{\geq 0} }H^0(G/P_1, \mathcal L_{2\varpi_1}^{\otimes d})^T$ is onto.\\

On the other hand we have $dim(U_{P_1}^-)= |\{\alpha \in R^+: \alpha \geq \alpha_1\}|=6$, where $U^-_{P_1}$ be unipotent radical of the  opposite parabolic subgroup of $P_1$ determind by $T$ and $B$.\\
 Hence the dimension of the affine cone $G/P_1$ is of dimension $7$.\\
So we have 

(4) the Krull dimension of $\oplus_{d \in \mathbb Z_{\geq 0}}H^0(G/P_1, \mathcal L_{2\varpi_1}^{\otimes d})^T$ is $3$. 
\\

From (3) and (4) we conclude that $\oplus_{d \in \mathbb Z_{\geq 0}}H^0(G/P_1, \mathcal L_{2\varpi_1}^{\otimes d})^T$ is a polynomial ring.
\\
Hence $\oplus_{d \in \mathbb Z_{\geq 0}}H^0(G/B, \mathcal L_{2\varpi_1}^{\otimes d})^T$ is a polynomial ring.
\\

Now we prove that $dim(H^0(G/B , \mathcal L_{2\varpi_1})^T) \leq$ rank of $G$ .\\

The $T$ invariant monomials in $Sym^2((\mathbb C^{8})^*)$ are of the form $X_iX_j$ for $1\leq i \leq 4, j=9-i$.\\
Hence we have $dim(Sym^2((\mathbb C^{8})^*)^T)$ is $4$.\\
Since $Sym^2((\mathbb C^{8})^*)=V(2\varpi_1)^*+\mathbb C q$, $dim(V(2\varpi_1)^T)=3$.\\
Hence $dim(H^0(G/B , \mathcal L_{2\varpi_1})^T) \leq $ rank of $G$ .\\

\underline{Type $D_n,n\neq 4$:}

By theorem 4.2 of \cite{r6}, the coxeter element $w$ for which there is dominant
character such that $X(w)_{T}^{ss}(L_{\chi})$ is non empty is $s_ns_{n-1}\cdots s_2s_1$.
The indicomposable dominant character with this property is $\chi=2\varpi_1=2(\alpha_1+\cdots+\alpha_{n-2})+\alpha_{n-1}+\alpha_n$.
\\ 
Proof in this case is similar to that of $\chi = 2\varpi_{1}$ in type $B_{n}$.\\

\underline{For other types :}

By theorem 4.2 of \cite{r6}, there are no coxeter element $w$ and dominant
character $\chi$ such that $X(w)_{T}^{ss}(L_{\chi})$ is non empty.

\end{proof}


\begin{thebibliography}{22}

 \bibitem[1]{r1} C.Chevalley, Invariants of Finite Groups Generated by Reflections, Amer. J. Math. 77(1955), 778-782.
\bibitem[2]{r2} J. E. Humphreys, Conjugacy classes in semisimple algebraic groups,  Math. Surveys Monographs, vol. 43, Amer. Math. Soc.
\bibitem[3]{r3} J.E. Humphreys, Introduction to Lie algebras and representation theory, Springer, Berlin Heidelberg, 1972.
\bibitem[4]{r4} J. E. Humphreys, {\it Linear Algebraic Groups}, Springer-Verlag, 1975.
\bibitem[5]{r5} S. S. Kannan, Torus quotients of homogeneous spaces-II, Proc. Indian Acad. Sci.(Math. Sci), 109(1999), no 1, 23-39.
\bibitem[6]{r6} S. S. Kannan, S.K.Pattanayak, Torus quotients of homogeneous spaces-minimal dimentional Schubert varieties admitting semi-stable points, Proc.Indian  Acad.Sci.(Math.Sci), 119(2009), no.4, 469-485.
\bibitem[7]{r7} V.Lakshmibai, C. Musili and C.S Seshadri, Cohomology of line bundles on $G/B$, Ann.Sc.\'{E}c.Norm.Sup.t.7  89 (1974).
\bibitem[8]{r8} D. Mumford, J. Fogarty and F. Kirwan, Geometric Invariant theory,
(Third Edition), Springer-Verlag, Berlin Heidelberg, New York, 1994. 
\bibitem[9]{r9} P.E. Newstead, Introduction to Moduli Problems and Orbit 
Spaces, TIFR Lecture Notes, 1978.
\bibitem[10]{r10} J.P.Serre, Groupes finis d'automorphisms d'anneaux locaux
reguliers, Colloq. d'Alg. Ecole Norm. de Jeunes Filles, Paris (1967), 1-11.
\bibitem[11]{r11} C.S. Seshadri, Introduction to Standard Monomial theory, 
Lecture notes No.4, Brandeis University, Waltham, MA, 1985.
\bibitem[12]{r12} G. C. Shephard and J. A. Todd, Finite Unitary Reflection 
Groups, Canadian J. Math. 6 (1954), 274-304.
\bibitem[13]{r13}  R. Steinberg, Regular elements of semisimple algebraic groups, Inst. des Hautes Etudes Sci. Publ. Math. 25 (1965), 49-80.
\bibitem[14]{r14} D.Wehlau, When is a ring of torus invariants a polynomial ring, Manuscripta Math.82,
161-170(1994).


\end{thebibliography}
\end{document}